\documentclass[reqno]{article}

\title{The Three 'R's and Dual Riordan Arrays}
\author{Thomas M. Richardson }
\date{February 2017}

\usepackage{amsmath}
\usepackage{amsfonts}
\usepackage{amsthm}
\usepackage[colorlinks=true]{hyperref}

\newtheorem{thm}{Theorem}[section]
\newtheorem{lem}{Lemma}[section]
\newtheorem{corr}{Corollary}[section]

\theoremstyle{definition}
\newtheorem{defin}{Definition}[section]
\everymath{\displaystyle} 

\newcommand{\convpow}[3]{#1_{#2}^{(#3)}}
\newcommand{\truncpow}[3]{#1_{#2}^{\left \{{#3} \right \}}}
\begin{document}

\maketitle

\begin{abstract}
\end{abstract}

\section{Introduction}
The three 'R's of the title are: reversion of series, recurrence relations, and Riordan arrays. In this paper we describe how the reversion of a series is related to convolutional recurrence relations for the series. We also show how the recurrence relation can be seen in the context of Riordan arrays. The connection with Riordan arrays includes a relation between reversion, inversion, and the A-sequence of Riordan arrays.

The motivation for doing this is to provide a clear statement of relationship between reversions and recurrences. Such a statement would have provided a helpful reference for the author's paper on the super Patalan numbers \cite[equation (14)]{SUPERPAT_PUBLISHED}.
As an example of the theorem that we prove in this paper, we also give a second convolutional recurrence for the Patalan numbers.

A second aspect of the relationship with Riordan arrays is that the doubly infinite matrix that is called an extended Super Patalan matrix in \cite{SUPERPAT_PUBLISHED}, is an example of a recursive matrix, as defined by Luz\'{o}n et. al. \cite{LMMS}. Luz\'{o}n et. al. also defined the concept of a dual Riordan array, and we describe the Riordan arrays and dual Riordan arrays that are related to the super Patalan numbers.

\section{Reversion of series}
\begin{defin}
Let $A(x) = \sum_{k \ge 0} a_k x^k$. The {\em reversion} of $A$ is a series $B(x) = \sum_{k \ge 0} b_k x^k$ such that $B(A(x)) = x$.
\end{defin}

Since polynomial multiplication corresponds to convolution of the coefficients, it is not terribly surprising that one can get recurrence relations involving convolution from reversion. One does have to be moderately careful in transforming the equation $B(A(x))=x$ into a recurrence relation on the coeffecients of $A$.

Following standard notation, we let $[x^n]P(x)$ be the coefficient of $x^n$ in the power series $P(x)$.
We observe that $B(A(x)) = \sum_{k \ge 0} b_k A(x)^k$.

\begin{lem}
\label{COEFFLEM}
Let $\convpow{a}{n}{k} = [x^n]A(x)^k$. Then
\begin{equation}
\label{COEFFEQN}
\convpow{a}{n}{k} = \sum_{i_1+\dots+i_k=n}\prod_{j=1}^k a_{i_j}.
\end{equation}
\end{lem}

The summation in equation \eqref{COEFFEQN} is over all degree sequences of length $k$ that sum to $n$.

\begin{lem}
If $B(x)$ is the reversion of the series $A(x)$, then $B(a_0) = 0$.
\end{lem}

\begin{lem}
\label{B1LEM}
If $B(x)$ is the reversion of the series $A(x)$, then $B'(a_0) = 1/a_1$.
\end{lem}

\begin{proof}
The equation $B(A(x))=x$ implies
$[x]B(A(x)) = 1$, and this implies
\begin{equation}
\label{BPRIMEEQ1}
\sum_{k \ge 0} b_k \convpow{a}{1}{k} = 1. 
\end{equation}
Now $\convpow{a}{1}{k} = k a_0^{k-1}a_1$, so
equation \eqref{BPRIMEEQ1} is equivalent to
\begin{equation}
\label{BPRIMEEQ2}
\sum_{k > 0}b_k k a_0^{k-1} a_1 = 1. 
\end{equation}
Now $\sum_{k > 0}b_k k a_0^{k-1} = B'(a_0)$, so equation \eqref{BPRIMEEQ2} is equivalent to
\begin{equation*}
B'(a_0) a_1 = 1. 
\end{equation*}
\end{proof}

\section{From reversion to recurrence relations}
In this section, we give two theorems that describe a convolutional recurrence relation for any power series with a reversion that is a polynomial. One applies to any such power series, and the other to power series with constant term 0. 

\begin{lem}
Let $(a_n)$ be a sequence with generating function $A(x)$.  Let $\convpow{a}{n}{k}$ be as in Lemma \ref{COEFFLEM}. Then $\convpow{a}{n}{k}$ satisfies
$\convpow{a}{n}{1}=a_n$ and
\begin{equation}
\label{CONVEQN}
\convpow{a}{n}{k}=\sum_{j=0}^n a_{n-j}\convpow{a}{j}{k-1}.
\end{equation}
\end{lem}

We next define a term that is closely related to $\convpow{a}{n}{k}$.
\begin{defin}
Let $(a_n)$ be a sequence. Define $\truncpow{a}{n}{k}$ by
$\truncpow{a}{n}{1}=0$ and, for $k>1$, 
\begin{equation}
\label{TRUNCEQN}
\truncpow{a}{n}{k}=\sum_{j=1}^{n-1} a_{n-j}\convpow{a}{j}{k-1}+a_0\truncpow{a}{n}{k-1}.
\end{equation}
\end{defin}

\begin{lem}
\label{CONVTRUNCLEM}
The terms $\convpow{a}{n}{k}$ and $\truncpow{a}{n}{k}$ satisfy
\begin{equation}
\label{TRUNCCONVEQN}
\convpow{a}{n}{k}=\truncpow{a}{n}{k}+k a_0^{k-1}a_n
\end{equation}
\end{lem}
\begin{proof}
We use induction on $k$.
For k = 1 we have $\convpow{a}{n}{1}=a_n$, $\truncpow{a}{n}{1}=0$, and $k a_0^{k-1}a_n=a_n$.
For $k>1$, by equation \eqref{CONVEQN}, we get
\begin{equation}
\label{TCEQN1}
\convpow{a}{n}{k}=a_n\convpow{a}{0}{k-1} + a_0\convpow{a}{n}{k-1} + \sum_{j=1}^{n-1}a_{n-j}\convpow{a}{j}{k-1}.
\end{equation}
By induction, $\convpow{a}{n}{k-1} = \truncpow{a}{n}{k-1} + (k-1)a_0^{k-2}a_n$.
Thus equation \eqref{TCEQN1} is equivalent to
\begin{equation}
\label{TCEQN2}
\convpow{a}{n}{k}=
a_0\bigl(\truncpow{a}{n}{k-1}+(k-1)a_0^{k-2}a_n\bigr) + a_n\convpow{a}{0}{k-1} + \sum_{j=1}^{n-1}a_{n-j}\convpow{a}{j}{k-1}.
\end{equation}
Now $\convpow{a}{0}{k-1} = a_0^{k-1}$, 
so we collect the multiples of $a_0^{k-1}a_n$ in \eqref{TCEQN2} to get
\begin{equation}
\label{TCEQN3}
\convpow{a}{n}{k}
= a_0\truncpow{a}{n}{k-1}+ka_0^{k-1}a_n + \sum_{j=1}^{n-1}a_{n-j}\convpow{a}{j}{k-1}.
\end{equation}
Now equation \eqref{TRUNCEQN} lets us substitute $\truncpow{a}{n}{k}$ for the first and third terms in \eqref{TCEQN3}, giving the desired result.
\end{proof}

\begin{corr}
\label{CONVEQTRUNC}
In the situation of Lemma \ref{CONVTRUNCLEM}, if $a_0=0$, then $\convpow{a}{n}{k}=\truncpow{a}{n}{k}$ for $k>1$.
\end{corr}

Next we consider the equation $0 = [x^n]B(A(x))$ for $n>1$.
\begin{thm}
The coefficient $a_n$, for $n>1$, satisfies
\begin{equation}
\label{REVREC}
a_n = -a_1\sum_{k>0} b_k\truncpow{a}{n}{k}.
\end{equation}
\end{thm}

\begin{proof}
By Lemma \ref{CONVTRUNCLEM}, we have 
\begin{align}
0 =&[x^n]B(A(x)) \\
=& \sum_{k>0} b_k\truncpow{a}{n}{k} \\
=& \sum_{k>0} b_k k a_0^{k-1}a_n 
+ \sum_{k>0} b_k\truncpow{a}{n}{k}. \label{RECTHMEQN1}
\end{align}

By Lemma \ref{B1LEM}, we have $\sum_{k>0} b_k k a_0^{k-1} = 1/a_1$.
so the first term is equal to $a_n/a_1$.
Substituting this into \eqref{RECTHMEQN1} gives
\begin{equation*}
0=a_n/a_1 
+ \sum_{k>0} b_k\truncpow{a}{n}{k}.
\end{equation*}
Now solving for $a_n$ gives the desired result.
\end{proof}

\begin{corr}
\label{ZEROCOROLLARY}
In the situation of Lemma \ref{CONVTRUNCLEM}, if $a_0=0$, then for $n>1$ we have
\begin{equation}
\label{REVREC2}
a_n = -a_1\sum_{k>1} b_k\convpow{a}{n}{k}.
\end{equation}
\end{corr}

\begin{proof}
This follows from the fact that $\truncpow{a}{n}{1}=0$ and Lemma \ref{CONVEQTRUNC}.
\end{proof}

\begin{corr}
\label{SHIFTCOROLLARY}
Let $A(x)=\sum_{k\ge 0}a_nx^n$ be a power series with $a_0\ne 0$, let $C(x)=xA(x)$, and let $B(x)$ be the reversion of $C(x)$.Then 
\begin{equation}
\label{CONVREC_shift}
a_n = -\sum_{k>1}  b_k\convpow{a}{n-k+1}{k}.
\end{equation}
\end{corr}

\begin{proof}
Since $C(x) = xA(x)$, the coefficients of $C$ and $A$ are related by $c_n=a_{n-1}$ for $n>0$ and $c_0=0$.
By Corollary \ref{ZEROCOROLLARY}, we have
\begin{equation}
\label{RECURC}
c_n = -c_1\sum_{k>1} b_k\convpow{c}{n}{k}.
\end{equation}
Substituting in terms of the coefficients of $A$ we get
\begin{equation}
\label{RECURA}
a_{n-1} = -a_0\sum_{k>0} b_k\convpow{a}{n-k}{k}.
\end{equation}
Note that the subscripts of the corresponding terms $\truncpow{c}{n}{k}$ and $\truncpow{a}{n-k}{k}$ 
in equations \eqref{RECURC} and \eqref{RECURA} differ by $k$ units. Every term in the corresponding sums has $k$ factors, and every factor's subscript differs by $1$, so the subscripts of their products differ by $k$ units.
Now shifting the index on $(a_n)$ by one unit gives
\begin{equation}
a_{n} = -a_0\sum_{k>0} b_k\convpow{a}{n-k+1}{k}.
\end{equation}
\end{proof}

In general, equation \eqref{REVREC} does not give us a way to compute the terms of the sequence $(a_n).$ It reduces to a finite sum if $B(x)$ is a polynomial. In that case, $a_0$ must be a root of $B(x)$, and $a_1$ is determined by Lemma \ref{B1LEM}. Then by equation (5), for a given $n$, $\truncpow{a}{n}{k}$ for $k > 1$ is computed before $a_n$, equation \eqref{TRUNCEQN} gives a recursive formula.  In the case $a_0=0$, the sum terminates for each $n$, but the number of terms grows with $n$. We can always get recurrence relations with either a finite sum or terminating sums by using the reversion of $xA(x)$ when $a_0 \ne 0$. We consider some examples next.

\section{Examples 1: Catalan and Patalan recurrences with non-zero constant term}
We consider convolutional recurrences for the Catalan and Patalan numbers.
Let $(a_n)_{n\ge 0}$ be the Patalan numbers of order $p$, and let $A(x)$ be their generating function.
Define the sequences $(c_n)$ and $(d_n)$ by $c_0=0$, $c_n = d_n = a_{n-1}$ for $n>0$, and $d_0=-1/p$. Also let $C(x)$ and $D(x)$ be their respective generating functions.
Thus $C(x) = xA(x)$, and $D(x)=C(x)-1/p$.

Now $C(x)$ is the reversion of
$f(x) \frac{1 - (1-px)^p}{p^2}$. This expression for $f$ expands to 
$f(x) = -\sum_{k=1}^p \binom{p}{k} p^{k-2}(-x)^k$ \cite{SUPERPAT_PUBLISHED}. 
By equation \eqref{REVREC}, we get the convolutional recurrence
\begin{equation}
\label{CONVREC_0}
c_n = \sum_{k\ge 2} (-1)^k\binom{p}{k} p^{k-2}\convpow{c}{n}{k}.
\end{equation}
By Corollary \ref{SHIFTCOROLLARY}, this gives the convolutional recurrence
\begin{equation}
\label{CONVREC_a}
a_n = \sum_{k>1} (-1)^k\binom{p}{k} p^{k-2}\convpow{a}{n-k+1}{k}.
\end{equation}
Equation \eqref{CONVREC_a} is equivalent to equation (14) of \cite{SUPERPAT_PUBLISHED}.

Next we consider $d(n)$ and $D(x)$. Since $D(x) = C(x)-1/p$, and $C(x)$ is the reversion of $\frac{1 - (1-px)^p}{p^2}$, we see that $D(x)$ is reversion of \begin{equation*}
g(x)=\frac{1 - \bigl(1-p(x+1/p)\bigr)^p}{p^2} = \frac{1 - (-px)^p}{p^2}.
\end{equation*}
This gives the convolutional recurrence
\begin{equation}
\label{CONVREC_p}
d_n = (-1)^{p} p^{p-2}\truncpow{d}{n}{p}.
\end{equation}
While the convolutional recurrences of equations \eqref{CONVREC_0} and \eqref{CONVREC_a} are essentially the same except for the indexing, the recurrence relation of equation \eqref{CONVREC_p} is distinct when $p>2.$
For example, for $p=4$, we have 
$\truncpow{d}{2}{4} = \frac{3}{8}$, 
$\truncpow{d}{3}{4} = \frac{7}{2}$, and
$\truncpow{d}{4}{4} = \frac{77}{2}$. 
Thus the recurrence of equation \eqref{CONVREC_p} is not an integral recurrence.

The form of equation \eqref{CONVREC_p} perhaps is a more natural generalization of the well known recurrence relation for the Catalan numbers, in the sense that it defines the Patalan numbers of order $p$ in terms of a $p^{\text{th}}$ convolutional power. Of course, the compact notation $\truncpow{d}{n}{k}$ represents a more complex sum than in the convolutional recurrence relation for the Catalan numbers.

\section{Examples 2: reversions of small polynomials with constant term zero}
Most of the sequences in the OEIS that are related to reversions of small polynomials are listed as such in the OEIS \cite{OEIS}. The Catalan numbers are related to the reversion of $x-x^2$, corresponding to the well known convolutional formula for the Catalan numbers \cite[A000108]{OEIS}.

Other sequences that are related to reversions of small polynomials are listed in Table \ref{table:1}. The terms of these sequences may all be calculated using equation \eqref{REVREC2}.

\begin{table}
\caption{Reversions of low degree polynomials in OEIS}
\label{table:1}
\centering
\begin{tabular}{||c | c ||} 
 \hline
 Polynomial & Sequence of Reversion \\ [0.5ex] 
 \hline\hline
 $x-x^2-x^3$ & A001002 \\ 
 $x-2x^2-x^3$ & A192945 \\ 
 $x-x^2-2x^3$ & A250886 \\ 
 $x-3x^2-x^3$ & A120590 \\ 
 $x-2x^2-2x^3$ & A276310 \\ 
 $x-x^2-3x^3$ & A276314 \\ 
 $x-3x^2-2x^3$ & A276315 \\ 
 $x-2x^2-3x^3$ & A250887 \\ 
 $x-2x^2+x^3$ & A006013 \\ 
 $x-3x^2+x^3$ & A005159 \\ 
 $x-2x^2+2x^3$ & A085614 \\ 
 $x-4x^2+x^3$ & A276316 \\ [1ex] 
 \hline
\end{tabular}
\end{table}

\section{Recurrence relations and Riordan arrays}
The examples in the last section can be put into the context of Riordan arrays. A Riordan array is an infinite lower triangular matrix based on two power series. 

\begin{defin}
\label{RIORDANDEF}
Let $g(x)=\sum_{k\ge 0} g_kx^k$ and $f(x) = \sum_{k\ge 0} f_kx^k$ be power series, with $f_0=0$. Define the {\em Riordan array} $R=R(g,f)$ by $R_{n,k} = [x^n](g(x)f(x)^k)$, for integers $n,k\ge 0$.
\end{defin}

The elements of the Riordan array may also be defined using two sequences, the $A$-sequence and the $Z$ sequence. The $A$-sequence depends only on $f$. The $A$ sequence is the coefficient sequence of the recurrence relation
\begin{equation}
d_{n+1,k+1}=\sum_{j=0}^\infty a_jd_n,k+j
\end{equation} for the entries of a Riordan array. See the paper of He and Sprugnoli for details \cite[equation (2.4)]{HESPRUGNOLI}. He and Sprugnoli show that the power series of the $A$-sequence satisfies $f(x) = xA(f(x))$ \cite[equation (2.6)]{HESPRUGNOLI}. If we let $\bar{f}(x)$ be the reversion of $f(x)$, this is equivalent to $A(x) = \frac{x}{\bar{f}(x)}$. This formula may be interepreted as a compostion of the inversion and reversion of the power series $f$. The following definitions make this precise.

\begin{defin} 
Let $(f_n)$, $n\ge 0$ be a sequence with $f_0 \ne 0$, and let $f(x) = \sum_n f_nx^n$. Define $INV((f_n))$ to be the sequence given by the inversion of $f(x)$.
\end{defin}

\begin{defin} 
Let $(f_n)$, $n\ge 0$ be a sequence with $f_0 \ne 0$, and let $F(x) = \sum_{n\ge 0} f_nx^{n+1}$. Let the power series $H(x) = \sum_{n\ge 0} h_nx^{n+1}$ be the reversion of $F(x)$. Define $REV((f_n))$ to be $(h_n)$.
\end{defin}

Now we can express the $A$-sequence in terms of $INV$ and $REV$.

\begin{thm}
\label{THREERS}
Let $(f_n)$, $n\ge 0$ be a sequence with $f_0 \ne 0$, and let $(a_n)$ be the $A$-sequence of the Riordan array $R(1, xf(x))$. Then $(a_n)=INV(REV((f_n)))$.
\end{thm}

Note that the $INV$ transform of a sequence defined by a linear recurrence gives the coefficients of the linear recurrence, while the $REV$ transform of a sequence defined by a convolutional recurrence gives the coefficients of the convolutional recurrence. From this view, Theorem \ref{THREERS} relates recurrences, reversion, and Riordan arrays.

\section{Recursive matrices and the extended super Patalan matrix}
The next definition, which extends definition \ref{RIORDANDEF}, follows Luz\'{o}n et. al \cite{LMMS}. to define doubly infinite Riordan Arrays.

\begin{defin}
\label{RIORDANDIDEF}
Let $g(x)=\sum_{k \ge 0} g_kx^k$ and $f(x) = \sum_{k \ge 0} f_kx^k$ be power series, with $f_0=0$ and $f_1\ne 0$. For $k<0$, let $f(x)^k$ be the multiplicative inverse of $f(x)^{-k}$ in the ring of formal Laurent series. Define the {\em recursive matrix} $D=D(g,f)$ by $D_{n,k} = [x^n](g(x)f(x)^k)$, for all integers $n,k$.
\end{defin}

The matrix $L$ in the authors paper The Super Patalan Numbers  is in fact a recursive matrix \cite[Theorem 6]{SUPERPAT_PUBLISHED}. In addition, $L$ is an involution. Riordan group involutions have been studied by Shapiro, Cheon, and Kim \cite{SHAPIROINVOLUTIONS}, \cite{CHEONKIM}, \cite{RIORDANINVOLUTION}.

The Super Catalan numbers, form the lower left quadrant of $L$ \cite[sequence A068555]{OEIS}. The ordinary Riordan array $R(1/\sqrt{1-4x},-x/(1-4x))$ forms the (absolute value of the) lower right quadrant of $L$ \cite[sequence A046512]{OEIS}. 

\section{The dual Riordan array}
We consider a Riordan array that is called the dual Riordan array by Luz\'{o}n et. al. \cite{LMMS}. It the upper left quadrant of a recursive matrix, after rotating and transposing.

We want to define the anti-transpose of a matrix, including infinite matrices, to be the reflection across the anti-diagonal. First, we will use the notation $A^@$ for the anti-transpose. The pronounciation of $@$ as "at" suggests its use as an abbreviation for "anti-transpose".

For finite matrices, the anti-transpose is just like the transpose, except it is the reflection across the anti-diagonal instead of the main diagonal. For doubly infinite matrices, we define $A^@_{i,j} = A_{-j,-i}$, for all integers $i,j$. For infinite matrices we use the same definition as for doubly infinite matrices, with the understanding that the index sets are as implied by the definition.

The anti-transpose of a recursive matrix is the $[0]$-complementary recursive matrix defined by 
Luz\'{o}n et. al. \cite[Definition 3.1]{LMMS}.

Now we define the Riordan dual $R^*(g,f)$.  

\begin{defin} Let $R(g,f)$ be a Riordan array as in definition \ref{RIORDANDEF}. 
Let $D(g,f)$ be the corresponding recursive matrix.
Define the {\em dual Riordan array} $R^*(g,f)$ by
\begin{equation}
R^*(g,f)_{i,j}=D(g,f)_{-j,-i}.
\end{equation}
\end{defin}

In the recursive matrix $D(g,f)$, the corresponding Riordan array $R(g,f)$ is the lower right quadrant, while the corresponding dual Riordan array $R^*(g,f)$ is the anti-transpose of the upper left quadrant.

In an earlier version of this paper, before we were aware of the work of Luz\'{o}n et. al., we used the term {\em doubly infinite Riordan array} for recursive matrix, and {\em Riordan dual} for dual Riordan array. 

The following theorem was proved by Luz\'{o}n et. al. \cite[Theorem 3.1]{LMMS}. The paper by Kruchinin and Kruchinin is also relevant to this theorem. \cite[Theorem 2]{KRUCHININ}
\begin{thm}
Let $g$ and $f$ be as in definition \ref{RIORDANDEF}.
Let $\bar{f}$ be the reversion of $f$. 
Let $\hat{f}$ be given by $\hat{f}(x)=\frac{x\bar{f'}(x)}{\bar{f}(x)}.$
Then the Riordan dual satisfies
\begin{equation}
R^*(g,f)=R(\hat{f}g(\bar{f}),\bar{f}).
\end{equation}
\end{thm}

We have seen that a recursive matrix contains both the corresponding Riordan array and the dual Riordan array. For the the example of the recursive matrix $L$ in \cite{SUPERPAT_PUBLISHED}, the matrix $S$ in the lower left quadrant of $L$ is a Super Catalan or Super Patalan matrix, with the column ordering reversed.
We state without proof the theorems that describe these relationships.

\begin{thm}
Let $p$ and $q$ be integers with $p\ge2$ and $0<q<p$. Let $Q(p,q)$ be the matrix of $(p,q)$-super Patalan numbers. Then $Q(p,q)$ forms the lower left quadrant of the recursive matrix
$D\biggl(\frac{1}{(1-p^2x)^{q/p}}, \frac{-x}{1-p^2x}\biggr).$
\end{thm}

See the comments about the generating functions of the columns of the doubly infinite matrix $E$ in the author's paper on the super Patalan numbers \cite{SUPERPAT_PUBLISHED}.

\begin{thm}
Let $p$ and $q$ be integers with $p\ge2$ and $0<q<p$. Let 
$g(x) = \frac{1}{(1-p^2x)^{q/p}}$, let 
$h(x) = \frac{1}{(1-p^2x)^{(p-q)/p}}$, and let 
$f(x) = \frac{-x}{1-p^2x}$.
Then the dual Riordan array
$R^*(g,f)=R(h,f)$.
\end{thm}

For $q=2$ and $p=1$, this shows that the Riordan array $R\biggl(\frac{1}{\sqrt{1-4x}}, \frac{-x}{1-4x}\biggr)$ is self-dual.
The Riordan array $R\biggl(\frac{1}{\sqrt{1-4x}}, \frac{-x}{1-4x}\biggr)$
is OEIS sequence \href{https://oeis.org/A046521}{A046521} \cite{OEIS}.

For one more example with $q=3$ and $p=1$, we find that the Riordan array $R\biggl(\frac{1}{(1-9x)^{1/3}} \frac{-x}{1-9x}\biggr)$
is dual to the Riordan array
$R\biggl(\frac{1}{(1-9x)^{2/3}}, \frac{-x}{1-9x}\biggr)$. These are OEIS sequences \href{https://oeis.org/A283150}{A283150} and \href{https://oeis.org/A283151}{A283151} \cite{OEIS}.

\end{document}